\documentclass[10pt, reqno]{amsart}
\usepackage{graphicx, amssymb, amsmath, amsthm, color, slashed, cite}
\numberwithin{equation}{section}

\usepackage{hyperref}

\let\Re=\undefined\DeclareMathOperator*{\Re}{Re}
\let\Im=\undefined\DeclareMathOperator*{\Im}{Im}

\newcommand{\R}{\mathbb{R}}

\newcommand{\eps}{\varepsilon}

\newtheorem{theorem}{Theorem}[section]

\newtheorem{lemma}[theorem]{Lemma}

\newtheorem{proposition}[theorem]{Proposition}

\theoremstyle{definition}

\theoremstyle{remark}

\newcommand{\qtq}[1]{\quad\text{#1}\quad}

\begin{document}

\title[Inhomogeneous NLS]{A simple proof of scattering for the intercritical inhomogeneous NLS}

\author{Jason Murphy}
\address{Department of Mathematics \& Statistics, Missouri S\&T}
\email{jason.murphy@mst.edu}

\begin{abstract} We adapt the argument of \cite{DM} to give a simple proof of scattering below the ground state for the intercritical inhomogeneous nonlinear Schr\"odinger equation.  The decaying factor in the nonlinearity obviates the need for a radial assumption.
\end{abstract}

\maketitle

\section{Introduction}

We revisit the problem of scattering below the ground state for the focusing, intercritical, inhomogeneous nonlinear Schr\"odinger equation (NLS).  We restrict our attention to the case of a cubic nonlinearity in three dimensions, i.e.
\begin{equation}\label{nls}
(i\partial_t + \Delta) u + |x|^{-b}|u|^2 u=0,\quad (t,x)\in\R\times\R^3,
\end{equation}
with the parameter $b$ chosen from the interval $(0,\tfrac12)$.  The scaling symmetry of \eqref{nls} identifies the equation as $\dot H^{s_c}$-critical, where $s_c=\tfrac{1+b}{2}\in(\tfrac12,\tfrac34)$. We call the equation \emph{intercritical} because the critical regularity $s_c$ lies between the special values $s_c=0$ and $s_c=1$, corresponding to the mass- and energy-critical cases, respectively.  Our restriction to the cubic nonlinearity serves primarily to simplify the presentation.  In particular, the argument presented here should also apply to more general powers and dimensions $d\geq3$. The restriction on $b$, which arises from applications of Hardy's inequality, could perhaps be relaxed by modifying or refining the arguments given below.

Denoting by $Q$ the ground state solution to
\begin{equation}\label{Qeqn}
\Delta Q - Q + |x|^{-b}Q^3 = 0
\end{equation}
and the conserved mass and energy of solutions by
\[
M(u) = \int |u|^2\,dx,\quad  E(u) = \int \tfrac 12|\nabla u|^2 - \tfrac14 |x|^{-b}|u|^4\,dx,
\]
we will prove the following.

\begin{theorem}\label{T} Let $0<b<\tfrac12$.  Suppose $u_0\in H^1(\R^3)$ obeys
\begin{equation}\label{sub1}
E(u_0)^{1+b}M(u_0)^{1-b} < E(Q)^{1+b}M(Q)^{1-b}
\end{equation}
and
\begin{equation}\label{sub2}
\|\nabla u_0\|_{L^2}^{1+b}\|u_0\|_{L^2}^{1-b}< \|\nabla Q\|_{L^2}^{1+b} \|Q\|_{L^2}^{1-b}.
\end{equation}
Then the solution $u$ to \eqref{nls} with $u|_{t=0}=u_0$ is global in time and scatters, that is, there exist $u_\pm\in H^1$ such that
\[
\lim_{t\to\pm\infty}\|u(t)-e^{it\Delta}u_\pm\|_{H^1} = 0.
\]
\end{theorem}

Theorem~\ref{T} was established first in the radial setting by \cite{FG}, and subsequently in the non-radial setting by \cite{MMZ, CFGM}.  These works adopted the concentration-compactness approach to induction on energy pioneered in \cite{KM}, reducing the problem of scattering to the problem of precluding a global non-scattering solution that is below the ground state threshold in the sense of \eqref{sub1} and \eqref{sub2} and has precompact orbit in $H^1$. The preclusion of such a solution is achieved by using a localized virial argument.  In non-radial problems, compact solutions are typically parametrized by some moving spatial center $x(t)$.  The key to passing from the radial to the non-radial case for \eqref{nls} was the observation that the decaying factor in the nonlinearity already provides enough spatial localization to guarantee that $x(t)\equiv 0$, which in turn allows for a simple implementation of the virial argument.  The basic idea is that if $|x(t)|\to\infty$, then the solution would behave like an approximate solution to the \emph{linear} Schr\"odinger equation, contradicting the fact that it does not scatter.  Roughly speaking, the non-radial problem may be treated as if it were radial.

In this note we push this idea a bit further by showing that the argument of \cite{DM}, which gives a simple proof of scattering for the radial NLS, may be adapted to \eqref{nls} even in the non-radial case.\footnote{The paper \cite{XZ} similarly adapted the arguments of the related work \cite{ADM} to the $2d$ inhomogeneous NLS. However, the authors of \cite{XZ} continued to work in the radial setting.}  The argument of \cite{DM} has two ingredients: (i) a scattering criterion as in \cite{Tao} based on a `mass evacuation' condition, and (ii) a hybrid virial/Morawetz estimate as in \cite{TO}, which implies the mass evacuation condition for solutions below the ground state threshold.  The radial assumption is used in both steps to derive quantitative decay estimates at large radii via the radial Sobolev embedding estimate of \cite{Strauss}.  In both cases, however, this estimate is used only in controlling terms arising from the nonlinearity.  Observing that the decaying factor in the nonlinearity of \eqref{nls} already yields quantitative decay at large radii, we find that the simple argument of \cite{DM} suffices to treat \eqref{nls}, even in the non-radial case.

Without loss of generality, we consider scattering in the forward time direction only.  After collecting a few preliminaries in Section~\ref{S:prelim}, we will prove the scattering criterion in Section~\ref{S:scatter} and the virial/Morawetz estimate in Section~\ref{S:virial}.  These two ingredients quickly imply Theorem~\ref{T}.

\subsection{Preliminaries}\label{S:prelim} We will need a few results related to well-posedness and scattering for \eqref{nls}.  We assume familiarity with the standard subcritical well-posedness theory for dispersive PDEs (e.g. the Duhamel formulation, Strichartz estimates, etc.).  Otherwise, we refer the reader to \cite{Cazenave} for a textbook treatment of nonlinear Schr\"odinger equations in general and to \cite{Guzman} for the specific case of the inhomogeneous NLS. 

For any initial datum $u_0\in H^1$, there exists a unique maximal-lifespan solution to \eqref{nls}.  Solutions conserve the mass and energy, and any solution that remains uniformly bounded in $H^1$ throughout its lifespan may be extended globally in time.  For such solutions we have the following local estimate:
\begin{equation}\label{holder}
\|u\|_{L_t^q H_x^{1,r}(I\times\R^3)} \lesssim (1+|I|)^{\frac1q} \qtq{for any Strichartz admissible pair}(q,r).
\end{equation}

We will also need the following small-data scattering result.
\begin{lemma}\label{small-data} Let $b\in(0,\tfrac12)$.  Suppose $u$ is a forward global solution to \eqref{nls} with $u|_{t=0}=u_0\in H^1$.  Suppose further that 
\[
\|u\|_{L_t^\infty H_x^1((0,\infty)\times\R^3)}= E\qtq{and} \|e^{it\Delta}u_0\|_{L_t^4 L_x^{\frac{6}{1-b}}((0,\infty)\times\R^3)}=  \eps.
\]
If $\eps$ is sufficiently small depending on $E$, then $u$ scatters in $H^1$ as $t\to\infty$. 
\end{lemma}

\begin{proof}[Sketch of proof] We choose a parameter $\rho\in(\tfrac{3}{b},\infty)$ and set $r=\tfrac{6\rho}{\rho-3}\in(6,\tfrac{6}{1-b})$.  We then define  
\[
\|u\|_{S} = \|u\|_{L_t^4 L_x^{r}}+\|u\|_{L_t^4 L_x^{12}}, 
\]
where here and below all space-time norms are taken over $(0,\infty)\times\R^3$.  By interpolation, Sobolev embedding, and Strichartz estimates, we can deduce that 
\[
\|e^{it\Delta}u_0\|_S \lesssim_E \eps^c
\]
for some $c>0$.  By Sobolev embedding, Strichartz estimates, H\"older's inequality, and Hardy's inequality, we can then estimate
\begin{align*}
\|u\|_S & \lesssim \eps^c + \| |x|^{-b}|u|^2 u\|_{L_t^2 H_x^{1,\frac65}} \\
& \lesssim \eps^c + \sum_{T\in\{1,\nabla,|x|^{-1}\}}\| |x|^{-b}u^2\,Tu\|_{L_t^2 L_x^{\frac65}} \\
& \lesssim \eps^c + \bigl\{ \| |x|^{-b}\|_{L^{\rho}(|x|>1)}+\| |x|^{-b}\|_{L^6(|x|\leq 1)} \bigr\}\|u\|_S^2\| u\|_{L_t^\infty H_x^1} \\
& \lesssim \eps^c + \|u\|_S^2 \|u\|_{L_t^\infty H_x^1}. 
\end{align*}
Thus for $\eps=\eps(E)$ sufficiently small, we derive $\|u\|_S\lesssim \eps^c$.  With this bound in hand, we may deduce that $e^{-it\Delta}u(t)$ is Cauchy in $H^1$ as $t\to\infty$ essentially by repeating the estimates above. \end{proof}

Next, we recall some properties of the ground state $Q$.  For more details, we refer the reader to  \cite[Theorems~1.1~and~1.2]{Farah}. 

The ground state $Q$ arises as an optimizer for the Gagliardo--Nirenberg inequality
\begin{equation}\label{GN}
\| |x|^{-b} |u|^4 \|_{L^1} \leq C_b \|u\|_{L^2}^{1-b} \|\nabla u\|_{L^2}^{3+b}.
\end{equation}
Using Pohozaev identities (obtained by multiplying \eqref{Qeqn} by $Q$ and $x\cdot\nabla Q$ and integrating by parts), one can connect the sharp constant to norms of $Q$ as follows:
\begin{equation}\label{Cb}
\|\nabla Q\|_{L^2}^{1+b} \|Q\|_{L^2}^{1-b} = \tfrac4{3+b} C_b^{-1} \qtq{and} E(Q)^{1+b}M(Q)^{1-b} = \tfrac{16}{(3+b)^{3+b}}(\tfrac{1+b}{2})^{1+b}C_b^{-2}. 
\end{equation}
Then, using \eqref{GN}, one can show that solutions obeying \eqref{sub1} and \eqref{sub2} are global in time and uniformly bounded in $H^1$, with
\begin{equation}\label{stay-below}
\sup_{t\in\R} \bigl\{\|\nabla u(t)\|_{L^2}^{1+b}\|u(t)\|_{L^2}^{1-b}\bigr\}<(1-\delta)\|\nabla Q\|_{L^2}^{1+b}\|Q\|_{L^2}^{1-b}\qtq{for some}\delta>0.
\end{equation}

The proof of scattering is then connected to the following virial identity:
\begin{equation}\label{pure-virial}
\tfrac{d}{dt}4\Im \int \bar u \nabla u\cdot x\,dx =  8\int |\nabla u|^2 - \tfrac{3+b}4|x|^{-b}|u|^4\,dx,
\end{equation}
which follows from \eqref{nls} and integration by parts. In particular, the bound \eqref{stay-below} and the sharp Gagliardo--Nirenberg inequality imply that the right-hand side of \eqref{pure-virial} is bounded below, yielding the monotonicity at the heart of scattering.  In practice, the presence of the weight $x$ in \eqref{pure-virial} necessitates spatial localization of the above identity, and accordingly we will need the following local form of coercivity.

\begin{lemma}\label{L:coercive} Let $u_0\in H^1$ satisfy the hypotheses of Theorem~\ref{T}, and let $u$ be the corresponding solution to \eqref{nls}.  There exists $\delta'>0$ so that for all $R$ sufficiently large,
\[
\int |\nabla[\chi_R u(t,x)]|^2 - \tfrac{3+b}4 |x|^{-b}|\chi_R u(t,x)|^4\,dx \geq \delta'\int |x|^{-b}|\chi_R u(t,x)|^4\,dx
\]
uniformly over $t\in\R$, where $\chi_R$ is a smooth cutoff to $|x|\leq R$. 
\end{lemma}

\begin{proof} (i) First suppose $\|\nabla f\|_{L^2}^{1+b}\|f\|_{L^2}^{1-b}<(1-\eta)\|\nabla Q\|_{L^2}^{1+b}\|Q\|_{L^2}^{1-b}$ for some $f\in H^1$ and $\eta\in(0,1)$.  Then  \eqref{GN} and \eqref{Cb} imply
\begin{align*}
\|f\|_{\dot H^1}^2 - \tfrac{3+b}{4}\||x|^{-1}f^4\|_{L^1} \geq \eta \|f\|_{\dot H^1}^2,
\end{align*}
which yields the following upon rearranging:
\[
\|f\|_{\dot H^1}^2 - \tfrac{3+b}{4}\||x|^{-b}f^4\|_{L^1} \geq \tfrac{\eta}{1-\eta}{\tfrac{3+b}{4}} \| |x|^{-b}f^4\|_{L^1}.
\]
(ii) Using (i), it suffices to to show that 
\begin{equation}\label{stay-below-local}
\sup_{t\in\R}\bigl\{\|\nabla[\chi_R u(t)]\|_{L^2}^{1+b} \|\chi_R u(t)\|_{L^2}^{1-b}\bigr\} < (1-\eta)\|\nabla Q\|_{L^2}^{1+b}\|Q\|_{L^2}^{1-b}
\end{equation}
for $R$ sufficiently large and some $\eta>0$.  As \eqref{stay-below} holds and  multiplication by $\chi_R$ only decreases the $L^2$-norm, it suffices to consider the $\dot H^1$-norm.  For this, we use 
\begin{equation}\label{benstrick}
\int \chi_R^2 |\nabla u|^2\,dx = \int |\nabla[\chi_R u]|^2 + \chi_R \Delta(\chi_R) |u|^2\,dx,
\end{equation}
which implies
\[
\| \nabla[\chi_R u]\|_{L^2}^2 \leq \|\nabla u\|_{L^2}^2 + \mathcal{O}(R^{-2} M(u)).
\]
We conclude that \eqref{stay-below-local} holds with $\eta=\tfrac12\delta$ for all $R$ sufficiently large. 
\end{proof}

\section{Scattering criterion}\label{S:scatter}

The first ingredient for the proof of Theorem~\ref{T} is the following scattering criterion as in \cite{Tao}.  For the standard NLS, this criterion is only valid in the radial setting.  As we will see, because of the decaying factor in the nonlinearity, this criterion is sufficient for \eqref{nls} even in the non-radial setting. 

\begin{proposition}\label{P:scatter} Let $b\in(0,\tfrac12)$.  Suppose $u$ is a global solution to \eqref{nls} obeying $\|u\|_{L_t^\infty H_x^1}\leq E$.  Then  there exist $\eps=\eps(E)>0$ and $R=R(E)>0$ so that if
\[
\liminf_{t\to\infty}\int_{|x|\leq R}|u(t,x)|^2\,dx \leq \eps^2,
\]
then $u$ scatters forward in time. 
\end{proposition}

\begin{proof} Throughout the proof, we allow implicit constants to depend on $E$.  With $\eps>0$ and $R>1$ to be chosen below, we first take $T>\eps^{-1}$ large enough that
\begin{equation}\label{smallatT}
\|e^{it\Delta} u_0\|_{L_t^4 L_x^{\frac{6}{1-b}}([T,\infty)\times\R^3)} < \eps \qtq{and}\int \chi_R(x)|u(T,x)|^2\,dx \leq \eps^2,
\end{equation}
where $\chi_R$ is a smooth cutoff to $\{|x|\leq R\}$.  The goal is then to prove
\[
\|e^{i(t-T)\Delta}u(T)\|_{L_t^4 L_x^{\frac{6}{1-b}}([T,\infty)\times\R^3)} < \eps^a \qtq{for some}a>0,
\]
which (for $\eps$ sufficiently small) implies scattering via Lemma~\ref{small-data}.  To estimate this norm, we rewrite the Duhamel formula for $u$ as follows:
\begin{align}
e^{i(t-T)\Delta}u(T) & = e^{it\Delta}u_0  +  i\int_{I_1} e^{i(t-s)\Delta}|x|^{-b}|u|^2 u(s)\,ds \label{D1}\\
& \quad + i\int_{I_2} e^{i(t-s)\Delta}|x|^{-b}|u|^2 u(s)\,ds, \label{D2} 
\end{align}
where $I_1=[T-\eps^{-c},T]$ and $I_2=[0,T-\eps^{-c}]$ for some $c>0$ to be specified below.

The linear term in \eqref{D1} is controlled in acceptable manner by \eqref{smallatT}, and hence it suffices to estimate the two integral terms. 

For the integral term in \eqref{D1}, we begin by using the triangle inequality, Sobolev embedding, and Strichartz estimates to obtain 
\begin{align}
\biggl\|\int_{I_1} & e^{i(t-s)\Delta}|x|^{-b}|u|^2 u(s)\,ds\biggr\|_{L_t^4 L_x^{\frac{6}{1-b}}([T,\infty)\times\R^3)}  \lesssim  \| |\nabla|^{\frac{1+b}{2}}\bigl[|x|^{-b}|u|^2 u\bigr]\|_{L_t^1 L_x^2(I_1\times\R^3)}.\label{D11}
\end{align}
We estimate this term by interpolating between $L_x^2$ and $\dot H_x^1$.  

We first consider the estimate in $L_x^2$.
 We begin by extending the small mass condition at $t=T$ in \eqref{smallatT} to the interval $I_1$. Using \eqref{nls} to derive the identity
 \[
 \partial_t |u|^2 = -2\nabla\cdot\Im(\bar u \nabla u),
 \]
we integrate by parts and use Cauchy--Schwarz to estimate
\[
\biggl| \tfrac{d}{dt} \int \chi_R(x)|u(t,x)|^2\,dx \biggr| \lesssim R^{-1}.
\]
With $R\geq\eps^{-2-c}$, this implies 
\[
\|\chi_R u \|_{L_t^\infty L_x^2(I_1\times\R^3)} \lesssim \eps.
\] 
Recalling that $b<\tfrac12$, we now choose an exponent $r=r(b)$ satisfying
\begin{equation}\label{r-conditions}
3<r<\tfrac{6}{1+2b}
\end{equation}
and $\theta=\theta(b)\in(0,1)$ satisfying
\begin{equation}\label{theta-conditions}
\theta<\min\{2(\tfrac1r-\tfrac{b}{3}),\tfrac{3}{r}-b-\tfrac12\}. 
\end{equation}
Writing $r_\theta$ for the solution to $\tfrac{1}{r}=\tfrac{\theta}{2}+\tfrac{1-\theta}{r_\theta}$, we use the triangle inequality, H\"older's inequality, Hardy's inequality, and Sobolev embedding to estimate
\begin{align*}
\| |x|^{-b}u\|_{L_x^r} & \lesssim \| |x|^{-b}(1-\chi_R)u\|_{L_x^r} + \|\chi_R u\|_{L_x^2}^{\theta}\| |x|^{-\frac{b}{1-\theta}} u\|_{L_x^{r_\theta}}^{1-\theta} \\
& \lesssim R^{-b} + \eps^{\theta} \| |\nabla|^{\frac{b}{1-\theta}+\frac{3}{2}-\frac{3}{r_\theta}}u\|_{L_x^2}^{1-\theta}
 \lesssim \eps^{\theta}
\end{align*}
uniformly over $t\in I_1$, where we have further imposed $R\geq\eps^{-\frac{\theta}{b}}$.  Here \eqref{r-conditions} and first constraint in \eqref{theta-conditions} guarantee that we may apply Hardy's inequality, while the second constraint in \eqref{theta-conditions} guarantees that the final norm is controlled by $H^1$.  Using H\"older's inequality, Sobolev embedding, and the local estimate \eqref{holder}, we therefore obtain
\begin{equation}\label{L2estimate}
 \begin{aligned}
\|  |x|^{-b} |u|^2 u\|_{L_t^1 L_x^2(I_1\times\R^3)}  &\lesssim \|x|^{-b}u\|_{L_t^\infty L_x^r(I_1\times\R^3)} \|u\|_{L_t^2 L_x^{\frac{3r}{r-3}}(I_1\times\R^3)}\|u\|_{L_t^2 L_x^6(I_1\times\R^3)}  \\
& \lesssim \eps^{\theta}|I_1| \lesssim \eps^{\theta-c}.
\end{aligned}
\end{equation}

We turn to the $\dot H_x^1$ estimate.  This leads to two terms, which take the form
\[
|x|^{-b}\mathcal{O}(u^2\nabla u) \qtq{and} \mathcal{O}(|x|^{-b-1}u^3). 
\] 
The first term may be estimated exactly as above; we simply put $\nabla u$ in $L_t^2 L_x^6$ instead of $u$ in \eqref{L2estimate}. For the second term, we instead use H\"older's inequality, Hardy's inequality, Sobolev embedding, and \eqref{holder} to estimate
\begin{align*}
\| |x|^{-b-1}u^3 \|_{L_t^1 L_x^2(I_1\times\R^3)} &\lesssim  \| |\nabla|^{\frac{b+1}{3}} u\|_{L_t^3 L_x^6(I_1\times\R^3)}^3 \\
& \lesssim \| |\nabla|^{\frac{b+2}{3}} u\|_{L_t^3 L_x^{\frac{18}{5}}(I_1\times\R^3)}^3 \lesssim \eps^{-c}.
\end{align*}
Here the application of Hardy's inequality requires $b<\tfrac12$. 

Returning to \eqref{D11}, we obtain the following bound by interpolation:
\[
\biggl\|\int_{I_1}  e^{i(t-s)\Delta}|x|^{-b}|u|^2 u(s)\,ds\biggr\|_{L_t^4 L_x^{\frac{6}{1-b}}([T,\infty)\times\R^3)}  \lesssim \eps^{\frac{(1-b)\theta}{2}-c},
\]
which (choosing $c=c(b)$ sufficiently small) is acceptable. 

It remains to estimate \eqref{D2} in $L_t^4 L_x^{\frac{6}{1-b}}$ on $[T,\infty)\times\R^3$.  Here the estimate is the same as in \cite{DM}.  We interpolate between the $L_t^4 L_x^3$-norm and the $L_t^4 L_x^\infty$-norm, using the identity
\[
 i\int_{I_2} e^{i(t-s)\Delta}|x|^{-b}|u|^2 u(s)\,ds = e^{i(t-T+\eps^{-c})\Delta}[u(T-\eps^{-c})-u_0]
\]
and Strichartz to obtain boundedness for the $L_t^4 L_x^3$-norm.  For the $L_t^4 L_x^\infty$-norm, we first use the dispersive estimate, Hardy's inequality, and Sobolev embedding to estimate
\begin{align*}
\biggl\| \int_{I_2} e^{i(t-s)\Delta}|x|^{-b}|u|^2 u(s)\,ds\biggr\|_{L_x^\infty} & \lesssim \int_{I_2} |t-s|^{-\frac32}\,ds\cdot  \| |x|^{-\frac{b}{3}}u \|_{L_t^\infty L_x^3}^3 \\
& \lesssim (t-T+\eps^{-c})^{-\frac12}\| |\nabla|^{\frac{b}{3}+\frac12}u\|_{L_t^\infty L_x^2}^3.
\end{align*} 
Thus the $L_t^4 L_x^\infty$-norm over $[T,\infty)$ is bounded by $\eps^{\frac{c}{4}},$ and hence we deduce the acceptable estimate
\[
\biggl\|\int_{I_2}  e^{i(t-s)\Delta}|x|^{-b}|u|^2 u(s)\,ds\biggr\|_{L_t^4 L_x^{\frac{6}{1-b}}([T,\infty)\times\R^3)}  \lesssim \eps^{\frac{(1+b)c}{8}}.
\]
\end{proof}

\section{Virial/Morawetz estimate}\label{S:virial}

In this section, we let $u$ be a solution to \eqref{nls} satisfying the hypotheses of Theorem~\ref{T}.  In particular, as discussed in Section~\ref{S:prelim}, $u$ is global, uniformly bounded in $H^1$, and obeys \eqref{stay-below}.  We will prove a virial/Morawetz estimate that implies the mass evacuation condition appearing in Proposition~\ref{P:scatter}.  

\begin{proposition}[Virial/Morawetz estimate]\label{P:morawetz} For any $T>0$ and $R>0$ sufficiently large,
\[
\tfrac{1}{T}\int_0^T \int_{|x|\leq R}|x|^{-b}|u(t,x)|^4\,dx\,dt \lesssim \tfrac{R}{T}+ \tfrac{1}{R^b}.
\]
\end{proposition}

\begin{proof} The proof is based off of the following identity, which follows from a direct computation using \eqref{nls} and integration by parts:  Given a smooth weight $a:\R^3\to\R$ and defining
\[
A_a(t) = 2\Im \int \bar u  u_j a_j \,dx, 
\]
where subscripts denote partial derivatives and repeated indices are summed, we have
\[
\tfrac{d}{dt}A_a = \int  4\Re a_{jk}\bar u_j u_k - |u|^2a_{jjkk}-|x|^{-b}|u|^4 a_{jj}-b|x|^{-b-2}|u|^4x_j a_j\,dx.
\]

Inspired by \cite{TO}, we choose a weight that interpolates between the standard virial and Morawetz weights.  In particular, choosing $R$ sufficiently large as in Lemma~\ref{L:coercive}, we let $a$ be a radial function satisfying
\[
a(x) = |x|^2\qtq{for} |x|\leq \tfrac{R}{2}\qtq{and} a(x)=2R|x|\qtq{for}|x|>R.
\]
For $\tfrac{R}{2}<|x|\leq R$ we impose 
\[
\partial_r a \geq 0,\quad \partial_r^2 a\geq 0,\qtq{and}|\partial^\alpha a(x)|\lesssim_\alpha R|x|^{-|\alpha|+1}\qtq{for}|\alpha|\geq 1, 
\]
where $\partial_r$ denotes radial derivative. We observe that the conditions above imply nonnegativity of the matrix $a_{jk}$, and that we have the bound
\[
\sup_{t\in \R}|A_a(t)| \lesssim R\|u\|_{L_t^\infty H_x^1}^2 \lesssim R. 
\]

For $|x|\leq \tfrac{R}{2}$, we have
\[
a_j = 2x_j,\quad a_{jk} = 2\delta_{jk},\quad \Delta a = 6,\qtq{and} \Delta\Delta a = 0,
\]
while for $|x|>R$ we have
\[
a_j =\tfrac{2Rx_j}{|x|},\quad  a_{jk}=\tfrac{2R}{|x|}[\delta_{jk}-\tfrac{x_jx_k}{|x|^2}],\quad \Delta a = \tfrac{4R}{|x|},\quad \Delta\Delta a = 0. 
\] 
Thus, by the identities above,
\begin{align}
\tfrac{d}{dt} A_a & =  8\int_{|x|\leq \frac{R}{2}} |\nabla u|^2 - \tfrac{3+b}{4}|x|^{-b}|u|^4\,dx \label{Mor1} \\
& \quad + \int_{|x|>R} \tfrac{8R}{|x|}|\slashed{\nabla}u|^2 -\tfrac{2R(2+b)}{|x|}|x|^{-b}|u|^4 \,dx \label{Mor2}  \\
& \quad + \int_{\frac{R}{2}<|x|\leq R} 4\Re a_{jk}\bar u_j u_k + \mathcal{O}(R^{-b}|u|^4 + R^{-2}|u|^2)\,dx \label{Mor3},
\end{align} 
where $\slashed{\nabla}$ denotes the angular part of the derivative.

In \eqref{Mor1}, we insert $\chi_R^2$ and use the identity \eqref{benstrick}, Lemma~\ref{L:coercive}, and uniform $H^1$-boundedness of $u$ to obtain
\begin{align*}
\eqref{Mor1} & \geq 8\int |\nabla[\chi_R u]|^2 - \tfrac{3+b}{4}|x|^{-b}|\chi_R u|^4\,dx \\
& \quad - \mathcal{O}\biggl\{R^{-2}M(u)+\int[\chi_R^4 - \chi_R^2]|x|^{-b}|u|^4\,dx \biggr\}\\
& \geq \delta'\int |x|^{-b}|\chi_Ru|^4 \,dx - \mathcal{O}(R^{-b}).
\end{align*}

For \eqref{Mor2}, the angular derivative term is nonnegative, while the nonlinear term is estimated by $R^{-b}$. Similarly, in \eqref{Mor3} the first term is nonnegative while the second term is estimated by $R^{-b}$.  Note that in contrast to \cite{DM}, we do not use radial Sobolev embedding to obtain decay at large radii.  Instead, the decay comes directly from the nonlinearity. 

Applying the fundamental theorem of calculus on the interval $[0,T]$ now yields
\[
\int_0^T \int_{|x|\leq \frac{R}{2}} |x|^{-b} |u(t,x)|^4\,dx\,dt \lesssim R + TR^{-b}. 
\]
\end{proof}

\begin{proof}[Proof of Theorem~\ref{T}] Applying Proposition~\ref{P:morawetz} with $R\sim T^{\frac{1}{b+1}}$ and $T$ sufficiently large, we may find a sequence of times $t_n\to\infty$ and radii $R_n\to\infty$ such that 
\[
\lim_{n\to\infty} \int_{|x|\leq R_n} |x|^{-b}|u(t_n,x)|^4\,dx = 0. 
\]
Thus, given any $R>0$, we have by H\"older's inequality that
\[
\int_{|x|\leq R}|u(t_n,x)|^2\,dx \lesssim R^{\frac{3+b}{2}}\biggl(\int_{|x|\leq R}|x|^{-b}|u(t_n,x)|^4\,dx\biggr)^{\frac12} \to 0 \qtq{as}n\to\infty.
\]
We therefore derive scattering via Proposition~\ref{P:scatter}. \end{proof}

\end{document}